\pdfoutput=1
\documentclass[11pt]{article}
\usepackage{amssymb}
\usepackage{graphicx}
\usepackage{xcolor} 
\usepackage{tensor}
\usepackage{fullpage} 
\usepackage{amsmath}
\usepackage{amsthm}
\usepackage{verbatim}
\usepackage{hyperref}
\usepackage{enumitem}
\usepackage{mathrsfs} 
\setlist[enumerate]{leftmargin=1.5em}
\setlist[itemize]{leftmargin=1.5em}

\providecommand{\MR}{\relax\ifhmode\unskip\space\fi MR }

\providecommand{\href}[2]{#2}

\usepackage{seqsplit} 
\definecolor{green}{rgb}{0,0.8,0} 

\newtheorem{maintheorem}{Theorem}

\newtheorem{theorem}{Theorem}[section]

\newtheorem{lemma}[theorem]{Lemma}
\newtheorem{proposition}[theorem]{Proposition}

\theoremstyle{definition}

\theoremstyle{remark}
\newtheorem{remark}[theorem]{Remark}

\numberwithin{equation}{section}

\newcommand{\nnrm}[1]{{\vert\kern-0.25ex\vert\kern-0.25ex\vert #1 
    \vert\kern-0.25ex\vert\kern-0.25ex\vert}}

\newcommand{\supp}{{\mathrm{supp}}\,}

\newcommand{\ud}{\mathrm{d}}

\newcommand{\nb}{\nabla}

\newcommand{\tht}{\theta}

\newcommand{\bbN}{\mathbb N}

\newcommand{\bbR}{\mathbb R}

\newcommand{\bbT}{\mathbb T}

\newcommand{\bbZ}{\mathbb Z}

\begin{document}

\title{Velocity  blow-up in $C^1\cap H^2$ for the 2D Euler equations}

\author{
Min Jun Jo\thanks{Department of Mathematics, The University of British Columbia. E-mail: mjjo@math.ubc.ca} \and Junha Kim\thanks{School of Mathematics, Korea Institute for Advanced Study. E-mail: junha02@kias.re.kr}}
\date{\today}

\renewcommand{\thefootnote}{\fnsymbol{footnote}}

\renewcommand{\thefootnote}{\arabic{footnote}}


\maketitle


\begin{abstract}
We give a vorticity-dynamical proof of $C^1\cap H^2$-illposedness of the 2D Euler equations. Our construction shows that the unique Yudovich solution escapes both $C^1$ and $H^2$ instantaneously.
\end{abstract}


\section{Introduction}
\subsection{Main result}

We study the initial value problem for the Euler equations in $\bbR_+ \times \bbT^2 $:
\begin{equation}\label{Euler}
    \begin{cases}
        \partial_t u + u\cdot \nb u + \nb p=0, \\
        \operatorname{div} u=0,\\
        u(\cdot,0)=u_0,
    \end{cases}
\end{equation}
where $u_0$ is the initial data satisfying $\operatorname{div}u_0=0$. We establish the following theorem.
\begin{maintheorem}\label{thm1}
    There exists an initial $u_0 \in C^1 \cap H^2(\bbT^2)$ such that the unique Yudovich solution $u$ satisfies
\begin{equation*}
      \sup_{t\in(0,\delta)}\|u(t,\cdot)\|_{C^{1}}=\sup_{t\in(0,\delta)}\|u(t,\cdot)\|_{H^2}=\infty,  
\end{equation*}
for any $\delta>0$.
\end{maintheorem}

\begin{remark}
    At the vorticity level, our initial data is the sum of a \emph{balanced} target $m$ and a \emph{stirring} piece $h$:
\begin{equation}\label{eq_initial}
        \omega_0 (r,\theta)= m(r,\theta) + h(r,\theta),
\end{equation}
where $m$ is both periodic and mean-zero in $\theta$ while the sign of $h$ is one-sided on $[0,1]^2$. For more properties of our initial data, see Section~\ref{sec_initial} and Section~\ref{sec_proof}. Such a choice is inspired by \cite{EM}.
\end{remark}
\begin{remark}
    The inflation of $\|u(t,\cdot)\|_{C^1}$ occurs as the stirrer $h$ breaks down the symmetry of $m$.   Certain largeness of $h$ is required to serve as such a trigger.
\end{remark}

\subsection{Summary}
We summarize the contributions of our main result.
\begin{itemize}
    \item \textbf{Dual blow-up.} It is known that the 2D Euler equations are ill-posed in $C^1$ and $H^2$, respectively: see \cite{BL,EJ} for $H^2$, and see \cite{BL2,EM} for $C^1$. Our result establishes further ill-posedness in the finer space $C^1\cap H^2$, providing the simultaneous norm inflation in both $C^1$ and $H^2$.
    \item \textbf{Quantitative proof.} The finiteness of $\|u(t,\cdot)\|_{\operatorname{Lip}}$ is necessary to control the Lagrangian deformation induced by the gradient of the flow map $\Phi$. One may refer to Lemma 3.1 of \cite{BL2} and Section 3.2 of \cite{EM}. But the goal itself is to show the infiniteness of $\|u(t,\cdot)\|_{\operatorname{Lip}}$, which can be paradoxical for that Lagrangian approach. In the previous literature, such an issue was resolved via contradiction arguments. Our study delivers a quantitative proof with respect to time and the distance from the origin, letting us off the hook of proof by contradiction. 
    \item \textbf{Vorticity dynamics.} Our proof elucidates the non-hypothetical vorticity dynamics that occurs while the unique Yudovich solution escapes the target spaces. By the use of the key lemma in the next section, we measure how the vorticity is transported along the flow map. Indeed, we verify that each point vortex at the symmetric part $m$ in our initial data \eqref{eq_initial} runs through an approximate \emph{hyperbolic} scenario. 
    \item \textbf{Initially $C^1$ towards $H^2$ blow-up} As alluded to before, the size of $\|u(t,\cdot)\|_{\operatorname{Lip}}$ is inherently important in Lagrangian formulation.  In the preceding studies \cite{BL,EJ,JY}, the primary mechanism of $H^2$ blow-up was initial largeness of velocity gradients, namely $\|\nb u_0\|_{L^\infty}=\infty$. This immediately enforced the large Lagrangian deformation, leading to the $H^2$ blow-up.  Our work shows that there is a finer structure for causing the instantaneous $H^2$ blow-up without the largeness of $\nb u_0$.
    
\end{itemize}

\subsection{Background}

The Euler equations \eqref{Euler} are known to be well-posed in the Sobolev spaces $H^s$ for any $s>\frac{d}{2}+1$, where $d=2$ in our consideration. Heuristically, it follows from the a priori $H^s$ energy estimate
\begin{equation*}
    \frac{\ud}{\ud t} \|u\|_{H^s}^2 \lesssim \|\nb u\|_{L^\infty} \|u\|_{H^s}^2,
\end{equation*}
which can be obtained via Gagliardo-Nirenberg inequalities for integer $s$ or Kato-Ponce type commutator estimates \cite{KP} for non-integer $s$. If $s> \frac{d}{2}+1$, then the Sobolev embedding from $H^s$ to $W^{1,\infty}$ works and so the energy estimate can be closed. The case $s>\frac{d}{2}+1$ is called subcritical, and the case $s=\frac{d}{2}+1$ is called critical for the Euler equations; $H^2$ is a critical Sobolev space for \eqref{Euler}. It turned out that the 2D Euler equations were strongly ill-posed in $H^2$; see \cite{BL}. 

A remarkable feature of \cite{BL} was the use of the vorticity formulation for \eqref{Euler} in view of Lagrangian deformation associated with the flow map. Later, the proof of \cite{BL} was simplified in \cite{EJ} where a certain kind of \emph{key lemma} played a crucial role. The key lemma made its first appearance in \cite{KS} where the double-exponential growth of the vorticity gradient was shown. See Lemma 3.1 in \cite{KS} and Lemma 1 in \cite{EJ}.  Leveraging the odd symmetry of the vorticity, the key lemma gives an effective vorticity-dynamical description of the velocity via the Biot-Savart law. Aligned with the particular point of view related to vorticity formulation, many researchers considered the generalized SQG equation
\begin{equation}\label{gSQG}
\begin{cases}
    \partial_t \omega + u \cdot \nb 
    \omega = 0, \\
    u=\nb^{\perp}(-\Delta)^{-1+\alpha}  \omega,
\end{cases}  
\end{equation}
where $\alpha$ ranges from $0$ to $1/2$. One can check that taking $\nb^{\perp}=(-\partial_2,\partial_1)$ for \eqref{Euler}, the case $\alpha=0$ of the generalized SQG equation corresponds to the 2D Euler equations. In the periodic settings, the exponential growth of  $\|\nb \omega\|_{L^\infty}$ for the Euler case $\alpha=0$ was proved in \cite{Zlatos}, and the exponential growth of $\|\nb^2 \omega \|_{L^\infty}$ for any $0<\alpha<1$ was shown in \cite{HK}. Both results employed the key lemmas, and especially multiple key lemmas were used in \cite{HK} due to the additional fractional dissipation in the Biot-Savart law. It is noteworthy that the proof of strong illposedness of the 2D Euler equations was further developed in \cite{JY} to be fully quantitative, without using any contradiction argument. 

Another well-posedness theory of crucial importance is Yudovich theory based on the $L^\infty$ vorticity. In the two-dimensions, any $L^p$ norm of the vorticity $\omega$ of \eqref{gSQG} with $p\in[1,\infty]$ is conserved in time, fluid particles being transported along the flow map induced by the velocity field. The target of this paper is \eqref{Euler} in the periodic setting, which is the $\alpha=0$ case of \eqref{gSQG}. In such a case, the velocity $u$ satisfies the so-called log-Lipschitz continuity
\begin{equation}\label{logLip}
    |u(x)-u(y)|\leq C\|\omega_0\|_{L^\infty} |x-y|\log\left(1+\frac{1}{|x-y|}\right),
\end{equation}
knowing that $\|\omega(t,\cdot)\|_{L^\infty}=\|\omega_0\|_{L^\infty}$. For the proof of \eqref{logLip}, see \cite{MP}. The log-Lipschitz continuity is sufficient to guarantee the unique existence of the flow map $\Phi$ in \eqref{flow}, and so the solution triple $(u,\Phi,\omega)$ can be uniquely determined provided that $\|\omega_0\|_{L^\infty}<\infty.$ A natural question arises from \eqref{logLip} about whether $\|u\|_{\operatorname{Lip}}$ actually bears the logarithmic singularity, i.e., the log-Lipschitz continuity is sharp. Investigating $C^1$-ill/wellposedness of the 2D Euler equations is related to that question. One may notice further that, other than just $H^2$, any Sobolev spaces $W^{\frac{2}{p}+1,p}$ with $p\in(1,\infty]$ is critical for \eqref{Euler} in terms of the energy estimates, and the end case $p=\infty$ includes $C^1$. In \cite{BL2} and \cite{EM}, independently, it was proved that \eqref{Euler} is indeed illposed in $C^1$.

The purpose of this paper is to prove that the 2D Euler equations are illposed even in the finer space $C^1\cap H^2$, by showing that instantenously the Yudovich solution escapes both $C^1$ and $H^2$. The \emph{dual} blow-up phenomenon might suggest the delicacy of the Euler dynamics. 

As a final remark, the vorticity-descriptive nature of our proof allows us to quantify the actual dynamics. We partially realize the formal blow-up rate $\|\nb u(t,\cdot)\|_{L^\infty}\gtrsim \frac{1}{t}$ given in \cite{EJ2} in an indirect fashion during our proof: for any $t\in(0,\delta]$, there exists $t^{\ast}\in(0,t]$ such that $\|\nabla u(t^{\ast},\cdot)\|_{L^\infty(A_r)} \gtrsim \frac{1}{t}$ for $r \sim \exp(-\exp(\frac{C}{t}))$, where $A_r$ stands for the annulus between $|x|=\frac{r}{2}$ and $|x|=2r$. See \eqref{blowup} and the ensuing estimates. For the $H^2$ illposedness, we obtain the blow-up rate with respect to the distance from the origin as $\|\omega(t,\cdot)\|_{H^1(\bbT^2 \setminus B(0;r))}^2 \gtrsim (\log \frac{1}{r})^{\varepsilon}$ for some $\varepsilon>0$. We refer to \eqref{H2_dist} for details.

The rest of the paper is organized as follows. In Section 2, we introduce the notion of flow map and verify our key lemma. Section 3 is devoted to the proof of a robust existence theorem for a class of initial data candidates at the stream-functional level. The last section provides the proof of our main theorem.



\section{Preliminaries}

\subsection{Flow map}
When $\omega_0$ belongs to $L^\infty(\bbT^2)$, it is known that the velocity $u$ is log-Lipschitz and so the flow map $\Phi$ can be induced via the ordinary differential equation
\begin{equation}\label{flow}
    \partial_t \Phi (t,x) = u (t,\Phi(t,x))
\end{equation}
with the initial condition $\Phi(0,x)=x$ for any $x\in\bbT^2$. See \cite{MP}. For brevity, $\bbT^2=[-1,1)^2$.

\subsection{Key lemma}
 In this section, we state and prove a modified version of Lemma 2.1 in \cite{Zlatos}, which originates from Lemma 3.1 of \cite{KS}. Let $Q(x):=[x_1,1]\times [0,1]$ and $R(x):=[\frac{x_1}{4},x_1]\times [x_2,x_1]$ for any $x\in[0,1)^2$. The following is our key lemma we use frequently throughout the paper.
\begin{lemma}\label{lem_key}
Let $\omega(t,\cdot)\in L^\infty(\bbT^2)$ be odd in both $x_1$ and $x_2$. 
For any $x\in\bbT^2$ with $0<x_2<x_1<1/2$, there holds
\begin{equation}\label{key}
    u_j(t,x)=(-1)^{j}\left(\frac{4}{\pi} \int_{Q(2x)}\frac{y_1 y_2}{|y|^4}\omega(t,y)\,\mathrm{d}y+E_j(t,x)\right) x_j \quad\mbox{for}\quad j\in\{1,2\},
\end{equation}
where $E_j$'s satisfy the bounds
\begin{equation}\label{error_1}
    |E_1(t,x)| \leq C \|\omega(t,\cdot)\|_{L^\infty(\bbT^2)}
\end{equation} and
\begin{equation}\label{error_2}
    |E_2(t,x)| \leq C \|\omega(t,\cdot)\|_{L^\infty(\bbT^2)} + C\|\omega(t,\cdot)\|_{L^\infty(R(2x))} \log\left(1+\frac{x_1}{x_2}\right)
\end{equation}
for some universal constant $C>0$.
\end{lemma}

\begin{remark}
This excludes the logarithmic singularity in $E_1$, and specifies the region, $R(2x)$, where the singular behavior of $E_2$ can happen.
\end{remark}

\begin{proof}
 We treat $u_1$ first. Since $0<x_2<x_1,$ we can compute
\begin{equation}\label{est_Q}
\int_{2x_1}^{1}\int_{0}^{2x_2} \frac{y_1 y_2}{|y|^4}\,\mathrm{d}y \leq C \int_{x_1}^{1}\int_{0}^{x_2}\frac{z_1z_2}{|z|^4}\,\mathrm{d}z_2\,\mathrm{d}z_1 \leq C \int_{0}^{x_2}\frac{z_2}{x_1^2+z_2^2}\,\mathrm{d}z_2 \leq C \log\left(1+\frac{x_2}{x_1}\right) \leq C.
 \end{equation}
Combining the above observation with Lemma 2.1 in \cite{Zlatos} twice, we obtain \eqref{key} for $j=1$ and the error bound \eqref{error_1}, respectively.

For $u_2$, the Biot-Savart law says that $u_2(t,x)$ equals
\begin{equation*}
 \frac{2}{\pi}\sum_{n\in\bbZ^2}\int_{[0,1]^2}\left(\frac{(-x_1+y_1+2n_1)y_2(x_2-2n_2)}{|x-y-2n|^2|x-\bar{y}-2n|^2}-\frac{(-x_1-y_1+2n_1)y_2(x_2-2n_2)}{|x+y-2n|^2|x-\tilde{y}-2n|^2}\right)\omega(t,y)\,\mathrm{d}y,
\end{equation*}
due to the odd symmetry of $\omega(t,\cdot).$ The size of the above quantity with $n=(0,0)$ removed is bounded by $Cx_2\|\omega(t,\cdot)\|_{L^\infty(\bbT^2)}$, similarly to the proof of Lemma 2.1 in \cite{Zlatos}. Therefore, it suffices to consider the term corresponding to $n=(0,0),$ which is $x_2$ times
 \begin{equation*}
     \frac{2}{\pi} \int_{[0,1]^2}\left(\frac{y_2(y_1-x_1)}{|x-y|^2|x-\bar{y}|^2}-\frac{y_2(x_1+y_1)}{|x+y|^2|x-\tilde{y}|^2}\right)\,\omega(t,y)\,\mathrm{d}y.
 \end{equation*}
We may separate the integral into several regions; while the absolute values of the integrals over the other regions can be bounded by $C\|\omega(t,\cdot)\|_{L^\infty(\bbT^2)}$, the integral over $R(2x)=[\frac{x_1}{2},2x_1]\times [2x_2,2x_1]$ possibly bears the logarithmic singularity.

We treat the absolute value of the first integral only, since the second one can be dealt with similarly. On $R(2x),$ it is bounded by $\|\omega(t,\cdot)\|_{L^\infty(R(2x))}$ times
\begin{equation*}
C \int_{R(2x)} \frac{y_2(y_1-x_1)}{|x-y|^2|x-\bar{y}|^2}\,\mathrm{d}y \leq C \int_{0}^{x_1}\int_{x_2}^{1} \frac{z_2z_1}{|z|^4}\,\mathrm{d}z_2 \,\mathrm{d}z_1 \leq C \log\left(1+\frac{x_1}{x_2}\right).
\end{equation*} On the region $[0,\frac{x_1}{2}]\times[2x_2,2x_1]$, the absolute value of the integral is bounded by $\|\omega(t,\cdot)\|_{L^\infty(\bbT^2)}$ times
\begin{equation*}
\begin{split}  C \int_{x_2}^{1}\int_{\frac{x_1}{2}}^{x_1}\frac{z_1z_2}{(z_1^2+z_2^2)^2} \leq C \int_{\frac{x_1}{2}}^{x_1}\frac{z_1}{x_2^2+z_1^2}\leq C \log\left(\frac{x_1^2+x_2^2}{\frac{x_1^2}{4}+x_2^2}\right)\leq C.
\end{split}
\end{equation*}
In $[0,2x_1]\times[0,2x_2]$, the absolute value of the integral is bounded by $\|\omega(t,\cdot)\|_{L^\infty(\bbT^2)}$ times
\begin{equation*}
    C\int_{0}^{x_1}\int_{0}^{x_2}\frac{x_2z_1}{(z_1^2+z_2^2)(z_1^2+x_2^2)}\,\mathrm{d}z_2\,\mathrm{d}z_1 \leq C\int_{0}^{x_1}\frac{x_2}{z_1^2+x_2^2}\,\mathrm{d}z_1 \leq C.
\end{equation*}
For the integral over the remainder region $[0,2x_1]\times[2x_1,1]$, we first observe that $y_1\leq y_2$ for any $[0,2x_1]\times[2x_1,1].$ Writing $y_2=2x_1+\varepsilon$ with $\varepsilon>0$, we further observe that $y_2-x_2 \geq y_2-x_1= x_1+\varepsilon \geq x_1+\frac{\varepsilon}{2}=\frac{y_2}{2}$.
This leads to
\begin{equation*}
    \frac{1}{2}y_2 \leq |x-y|\leq 2y_2,\quad \frac{1}{2}y_2\leq |x-\bar{y}|\leq 2y_2.
\end{equation*}
Therefore, the 
absolute value of the corresponding integral is bounded by $\|\omega(t,\cdot)\|_{L^\infty(\bbT^2)}$ times
\begin{equation*}
C \int_{2x_1}^{1}\int_{0}^{2x_1} \frac{x_1}{y_2^3}\,\mathrm{d}y_1\,\mathrm{d}y_2 \leq C.
\end{equation*}
On the key region $Q(2x),$ we separate the region into $[2x_1,1]\times[2x_2,1]$ and $[2x_1,1]\times[0,2x_2].$ On the first region, we recall the proof of Lemma 2.1 in \cite{Zlatos}, which says
that the integral equals to
\begin{equation*}
    \int_{[2x_1,1]\times[2x_2,1]} \frac{y_1y_2}{|y|^4}\omega(t,y)\,\mathrm{d}y_1\,\mathrm{d}y_2 + B_2(t,x),
\end{equation*}
where $B_2(t,x)$ is bounded by $C\|\omega(t,\cdot)\|_{L^\infty(\bbT^2)}.$ Then, the previous computation \eqref{est_Q} reduces the proof to showing that the absolute value of the integral over $[2x_1,1]\times[0,2x_2]$ is bounded by $C\|\omega(t,\cdot)\|_{L^\infty(\bbT^2)}$. This is true because it is bounded by $\|\omega(t,\cdot)\|_{L^\infty(\bbT^2)}$ times
\begin{equation*}
   C\int_{0}^{x_2}\int_{x_1}^{1} \frac{z_1x_2}{(z_1^2+z_2^2)(z_1^2+x_2^2)} \,\mathrm{d}z_2\,\mathrm{d}z_1 \leq C \int_{0}^{x_2} \frac{x_2}{z_1^2+x_2^2}\,\mathrm{d}z_1 \leq C.
\end{equation*}
The proof is finished.
\end{proof}

\section{Stream-functional data}\label{sec_initial}
The specific choice of our initial data would be given in the last section. Here, instead, we prove a robust existence theorem for the general stream-function $\psi$ that satisfies the vorticity formulation $\Delta \psi(r,\theta) = f(r)g(\theta)$ in the polar coordinates.
\begin{proposition}\label{prop_initial}
       Let $\bar{f}:[0,1]\to [0,\infty)$ be a smooth concave function with
\begin{equation}\label{f_ass}
     \bar{f}(0) = \bar{f}(1) = 0 \quad \mbox{and} \quad r\bar{f}'(r) \to 0 \quad \mbox{as} \quad r \to 0.
\end{equation} 
Let $g:\bbR\to\bbR$ be a smooth $\frac{\pi}{2}$-periodic function whose mean value over the period is zero, i.e., \begin{equation}\label{g_ass}
     \int_{0}^{\frac{\pi}{2}}g(\theta)\,\mathrm{d}\theta=0.
 \end{equation} Then for any given continuous function $f : [0,1] \to \bbR$ with $|f| \leq \bar{f}$, there exists a stream-function $\psi\in C^2(\bbT^2)$ such that
       \begin{equation}\label{id_0}
           \Delta \psi (r,\theta) = f(r)g(\theta).
       \end{equation}
\end{proposition}
\begin{remark}
    There hold
    \begin{equation*}
    \begin{split}
    |\partial_{i}\partial_{j}\psi (r,\tht)| &\lesssim \sup_{\tau \in [0,r]} \bar{f}(\tau) + \sup_{\tau\in[0,r]} \tau \bar{f}'(\tau), \\
    |\partial_{i}\psi (r,\tht)| &\lesssim r \big(\sup_{\tau \in [0,r]} \bar{f}(\tau) + \sup_{\tau\in[0,r]} \tau \bar{f}'(\tau)\big), \\
    |\psi(r,\tht)| &\lesssim r^2 \big(\sup_{\tau \in [0,r]} \bar{f}(\tau) + \sup_{\tau\in[0,r]} \tau \bar{f}'(\tau)\big).
    \end{split}
    \end{equation*}  This says that $\omega_0=\Delta \psi$ and $\nb u_0=\nb \nb^{\perp} \psi$ can share the same decay rate near the origin.
\end{remark}
 \begin{proof} 
 For each $n\in\bbN,$ we set up the building block $\psi_n:(r,\theta)\mapsto \psi_n(r,\theta)$ as
 \begin{equation}\label{block}
 \begin{cases}
     \psi_{n}(r,\theta)=f_n(r)g_n(\theta),\\
     f_n(r)=-\frac{1}{r^{4n}}\int_{0}^{r}s^{8n-1}\int_{s}^{1}\tau^{1-4n}f(\tau)\,\mathrm{d}\tau\,\mathrm{d}s, \\ g_n(\theta)=a_n\cos{4n\theta}+b_n\sin{4n\theta},
 \end{cases}     
 \end{equation}
 where $a_n$ and $b_n$ are the $n$-th Fourier series coefficients of $g$ satisfying $a_n=\frac{4}{\pi}\int_{0}^{\pi/2}g(\theta)\cos{4n\theta}\,\mathrm{d}\theta$ and $b_n=\frac{4}{\pi}\int_{0}^{\pi/2}g(\theta)\sin{4n\theta}\,\mathrm{d}\theta$, respectively. The motivation of the above setting \eqref{block} is the set of the two simple observations:
\begin{equation}\label{id_1}
    \begin{split}
\Delta\Big(f_n(r)g_n(\theta)\Big) &= \left(f_n''(r)+\frac{1}{r}f_n'(r)-\frac{16n^2}{r^2}f_n(r)\right)g_n(\theta),
\\
    \frac{(r^{1-8n}(r^{4n}f_n(r))')'}{r^{1-4n}} &= f_n''(r)+\frac{1}{r}f_n'(r)-\frac{16n^2}{r^2}f_n(r).
    \end{split}
\end{equation}
Due to the choice of $f_n$ in \eqref{block}, we obtain that
\begin{equation}\label{id_2}
    \Delta\Big(f_n(r)g_n(\theta)\Big) = f(r)g_n(\theta).
\end{equation}
Since the inversion formula for the $\frac{\pi}{2}$-periodic mean-zero $g$ holds as $$g(\theta)=\sum_{n=1}^{\infty}(a_n\cos{4n\theta}+b_n\sin{4n\theta})=\sum_{n=1}^{\infty} g_n(\theta),$$
the observation \eqref{id_2} leads to the desired identity for $\psi:=\sum_{n=1}^{\infty}\psi_n$ as
\begin{equation*}
\Delta \psi(r,\theta)=\Delta\left(\sum_{n=1}^{\infty}\psi_n\right) (r,\theta) = f(r)\sum_{n=1}^{\infty}g_n(\theta) = f(r)g(\theta),
\end{equation*}
as long as we can justify the interchange of the Laplacian and the summation. Therefore, it suffices to show that 
$\psi(r,\theta)=\sum_{n=1}^{\infty}\psi_n (r,\theta)$ is well-defined and such $\psi$ belongs to $C^2(\bbT^2)$. This would imply that the identity \eqref{id_0} holds as claimed.

It remains to prove that such $\psi$ is well-defined and belongs to $C^2(\bbT^2)$. Define the three different $N$-sequences, say $S_N$, $S_{i,N}$, and $S_{ij,N}$, by
\begin{equation*}
\begin{split}
    S_N &= \sum_{n=1}^{N} \psi_{n}(r,\theta), \\
    S_{i,N} &= \sum_{n=1}^{N} \partial_i \psi_{n}(r,\theta), \\
    S_{ij,N} &= \sum_{n=1}^{N} \partial_i \partial_j \psi_{n} (r,\theta),
    \end{split}
\end{equation*}
where $\partial_i$ denotes the differentiation in the $i$-th variable. Note that the above sequences are well-defined for each $N\in\bbN$ because $f_n$ is twice continuously differentiable.

We will prove that the three sequences of summations are Cauchy sequences uniformly-in-$(r,\theta)$ converging to the corresponding limits. As a consequence, recalling Theorem 7.17 in \cite{Rudin}, we can justify the interchange between the limiting process and taking derivatives, guaranteeing the existence of the second derivatives of $\psi$. Then the uniform limit theorem would conclude that those derivatives are continuous, finishing the proof.

We start by simply observing that
\begin{equation*}
    \partial_i r = \frac{x_i}{r} \quad \mbox{and} \quad \partial_i \theta = \frac{(-1)^{i}x_{3-i}}{r^2}, \quad \forall i\in\{1,2\}.
\end{equation*}
Then we can estimate the radial terms as
\begin{align*}
    |f_n(r)|&\lesssim \frac{r^2}{n^2} \big(\sup_{\tau \in [0,r]} \bar{f}(\tau) + \sup_{\tau\in[0,r]} \tau \bar{f}'(\tau)\big), \\ \left|\partial_i \left(f_n(r)\right)\right| &\lesssim \frac{r}{n} \big(\sup_{\tau \in [0,r]} \bar{f}(\tau) + \sup_{\tau\in[0,r]} \tau \bar{f}'(\tau)\big),\\
    \left|\partial_i \partial_j \left(f_n(r)\right)\right| &\lesssim \sup_{\tau \in [0,r]} \bar{f}(\tau) + \sup_{\tau\in[0,r]} \tau \bar{f}'(\tau).
\end{align*}
We only show the third inequality because the others can be treated similarly. A direct computation gives \begin{equation}\label{ineq_second}
    \left| \partial_i \partial_j \left( f_n(r) \right) \right| \lesssim \left| \frac{n^2}{r^{4n+2}} \int_{0}^{r}s^{8n-1}\int_{s}^{1}\tau^{1-4n}f(\tau)\,\mathrm{d}\tau\mathrm{d}s \right| + \left| nr^{4n-2} \int_{r}^{1}\tau^{1-4n}f(\tau) \,\mathrm{d}\tau \right| + |f(r)|.
\end{equation} 
Using $\bar{f} \geq |f|$ and performing integration by parts with $\bar{f}(1)=0$, we see that
\begin{align*}
    \left| nr^{4n-2} \int_{r}^{1}\tau^{1-4n}f(\tau) \,\mathrm{d}\tau \right| &\leq nr^{4n-2} \int_{r}^{1}\tau^{1-4n}\bar{f}(\tau) \,\mathrm{d}\tau \\
    &= \frac{n}{4n-2} r^{4n-2} \left( r^{2-4n}\bar{f}(r) - \bar{f}(1)\right) + \frac{n}{4n-2} r^{4n-2} \int_{r}^{1} \tau^{2-4n} \bar{f}'(\tau) \,\ud \tau \\
    &\lesssim \bar{f}(r) + r^{4n-2} \int_{r}^{1} \tau^{2-4n} \bar{f}'(\tau) \,\ud \tau.
\end{align*} We note by $\bar{f}'(r) \geq \bar{f}'(\tau)$, $\tau \in [r,1]$ that
$$r^{4n-2} \int_{r}^{1} \tau^{2-4n} \bar{f}'(\tau) \,\ud \tau \leq r^{4n-2} \bar{f}'(r) \int_{r}^{1} \tau^{2-4n} \,\ud \tau = \frac{1-r^{4n-3}}{4n-3} r\bar{f}'(r),$$
which yields
$$ \left| nr^{4n-2} \int_{r}^{1}\tau^{1-4n}f(\tau) \,\mathrm{d}\tau \right| \lesssim \bar{f}(r) + \frac{1}{n} \sup_{\tau \in [0,r]} \tau \bar{f}'(\tau).$$
The above estimate further implies 
\begin{align*}
    \left| \frac{n^2}{r^{4n+2}}\int_{0}^{r}s^{8n-1}\int_{s}^{1}\tau^{1-4n}f(\tau)\,\mathrm{d}\tau\mathrm{d}s \right| &\lesssim \frac{n}{r^{4n+2}}\int_0^r s^{4n+1} \bar{f}(s) \,\ud s + \frac{1}{r^{4n+2}} \sup_{\tau \in [0,r]} \tau \bar{f}'(\tau) \int_0^r s^{4n+1}\,\ud s \\
    &\lesssim \sup_{\tau \in [0,r]} \bar{f}(\tau) + \frac{1}{n} \sup_{\tau \in [0,r]} \tau \bar{f}'(\tau).
\end{align*}
In light of \eqref{ineq_second}, we establish that
$$\left| \partial_i \partial_j \left( f_n(r) \right) \right| \lesssim \sup_{\tau \in [0,r]} \bar{f}(\tau) + \frac{1}{n} \sup_{\tau \in [0,r]} \tau \bar{f}'(\tau).$$
The angular terms can be bounded in a more direct fashion. We get
\begin{equation*}
   \left|\partial_i (\sin{4n\theta})\right|+  \left|\partial_i (\cos{4n\theta})\right| \lesssim \frac{n}{r}, \quad \mbox{and} \quad \left|\partial_i \partial_j (\sin{4n\theta)}\right|+ \left|\partial_i \partial_j (\cos{4n\theta)}\right|\lesssim \frac{n^2}{r^2}.
\end{equation*}
From the identity \eqref{block} combined with the above bounds, we deduce that
\begin{equation*}
    |\partial_i \partial_j \psi_n(r)| \lesssim \left(|a_n|+|b_n|\right) \left( \sup_{\tau \in [0,r]} \bar{f}(\tau) + \sup_{\tau \in [0,r]} \tau \bar{f}'(\tau) \right).
\end{equation*}
Due to the smoothness of $g$ in the angular variable $\theta$, the order of the decay of $a_n$ can be made arbitrarily large. In other words, $a_n \lesssim \frac{1}{n^\ell}$ for any $\ell\in\bbN$. This implies that
\begin{equation*}
    |S_{ij,N}-S_{ij,M}| = \left|\sum_{n=M+1}^{N} \partial_i \partial_j \psi_n \right| \lesssim \left( \sup_{\tau \in [0,r]} \bar{f}(\tau) + \sup_{\tau \in [0,r]} \tau \bar{f}'(\tau) \right) \sum_{n=M+1}^{N}\frac{1}{n^{\ell}} \lesssim \frac{1}{M^{\ell-1}}.
\end{equation*}
Therefore, the sequence $S_{ij,N}$ is uniformly-in-$(r,\theta)$ Cauchy with respect to $N\in\bbN$. We denote by $\psi_{ij}$ such point-wise limit of the sequence $\{S_{ij,N}\}_{N\in\bbN}$. In similar ways, we can prove that $S_{i,N}$ and $S_{i}$ uniformly converge to the limits, $\psi$ and $\psi_i$, respectively. Here $\psi$ agrees with the previous definition $\psi=\sum_{n=1}^{\infty}\psi_n$. This concludes that $\psi$, $\partial_i \psi$, and $\partial_{i}\partial_{j}\psi$ exist as the uniform limits of $S_N$, $S_{i,N}$, and $S_{ij,N}$, respectively. Since $S_N$, $S_{i,N}$, and $S_{ij,N}$ are twice continuously differentiable, continuously differentiable, and continuous, respectively for each $N$, the uniform limits are also twice continuously differentiable, continuously differentiable, and continuous, respectively. This is in light of the uniform limit theorem and Theorem 7.17 in \cite{Rudin}. This finishes the proof.
\end{proof}

\section{Dual blow-up in $C^1\cap H^2$}\label{sec_proof}
 To prove Theorem~\ref{thm1}, the core ingredients are the followings. 
 \begin{itemize}
     \item A specific choice of initial data
     \item Short-time hyperbolic dynamics
     \item A positivity lemma
 \end{itemize}
 We introduce and construct those pieces in the following subsections, and then put them together into a whole proof of Theorem~\ref{thm1} at the very end of this section.

\subsection{Choice of initial data}
We fix a smooth positive bump function $\phi:\bbR\to [0,\infty)$ satisfying
\begin{equation*}
    \phi(r)=\begin{cases}
        1 \quad \mbox{if}\,\,\, r\in\left(-\frac{1}{32},\frac{1}{32}\right),\\
        0 \quad \mbox{if}\,\,\, r \notin\left(-\frac{1}{16},\frac{1}{16}\right).
    \end{cases}
\end{equation*}
Setting up the tetradically-scaled bumps as $\phi^{(n)}(r) := \phi(4^n(r - 4^{-n}))$, we define the radial part $f$ by

$$f(r) = \sum_{n \geq N_0}n^{-\beta} \phi^{(n)}(r)$$
for some $N_0 \in \bbN$ and $\beta \in (\frac{1}{2},1]$. The numbers $N_0$ and $\beta$ will be specified later. 

For the angular part, we fix another smooth bump function $\varphi:[0,2\pi]\to[0,\infty]$ satisfying
\begin{equation*}
    \varphi(\theta)=\begin{cases}
        1 \quad \mbox{if}\,\,\, \theta\in\left(\frac{\pi}{3}+\frac{1}{3}\tht_0,\frac{\pi}{3}+\frac{2}{3} \tht_0 \right),\\
        0 \quad \mbox{if}\,\,\, \theta \notin\left(\frac{\pi}{3},\frac{\pi}{3}+\theta_0\right),
    \end{cases}
\end{equation*}
where $\tht_0>0$ will be taken later to be sufficiently small. Then we take our angular part as $$g(\tht) = \varphi(\tht) - \varphi(\frac{\pi}{2}-\tht).$$
Notice that $g$ can be naturally extended to the full angular interval $[0,2\pi]$ in a way that it remains  $\frac{\pi}{2}$-periodic and its mean-value over the period is always zero. Then the balanced part of our initial data is established as
$$m(r,\theta)=f(r)g(\theta), $$
in the polar coordinates.

As a final touch, we define a perturbation that would cause the breakdown of the symmetry of $m$ once the time evolution starts from the initial data. Fix the last bump function $h:[0,1]^2\to [0,\infty)$ such that

\begin{equation*}
    h(x)=\begin{cases}
        1 \quad \mbox{if}\,\,\, x \in \left[4^{-(N_0-3)},\frac{2}{3}\right]^2 \\
        0 \quad \mbox{if}\,\,\, x \notin\left(4^{-(N_0-2)},\frac{5}{6}\right)^2.
    \end{cases}
\end{equation*}
Then $h$ enjoys the non-degeneracy property $\int_{[0,1]^2} \frac{y_1y_2}{|y|^4} h(y) \,\ud y \geq \frac{1}{C}\| h \|_{L^{\infty}} N_0$ for some $C>0$. 

Now we gather up those pieces of initial data as
\begin{equation*}
    \omega_0 = m(r,\tht) + h(x).
\end{equation*}
Here, we briefly show that $u_0 = \nabla^{\perp} (-\Delta)^{-1} \omega_0 \in C^1(\bbT^2)$ for $\beta \in (0,1]$. We refer to \cite{EJ,JY} for one to prove $u_0 \in H^2(\bbT^2)$ with $\beta \in (\frac{1}{2},1]$. We recall the stream-function $\psi = \Delta^{-1} \omega_0 = \Delta^{-1}(m+h)$. Since $h \in C^{\infty}(\bbT^2)$ and $\int_{\bbT^2} h(x) \,\ud x = 0$, it holds $\nabla (-\Delta)^{-1} h \in H^s(\bbT^2)$ for any $s \in \bbR$. On the other hand, we can verify that $f$ and $g$ satisfy the assumptions for Proposition~\ref{prop_initial} by taking $\bar{f} = C (\log \frac{1}{r})^{-\beta}$ with some large $C>0$. Thus, we obtain  $(-\Delta)^{-1} m \in C^2(\bbT^2)$. 

Using that $\| u(t) \|_{L^{\infty}} \leq C \| \omega_0 \|_{L^{\infty}}$, we can take $\delta>0$ small enough to make $\supp h$ hardly move on the time interval $[0,\delta]$. Therefore, $\supp h \subset (4^{-(N_0-1)},1)^2$ and 
\begin{equation*}
    \frac{1}{C}\| h \|_{L^{\infty}} N_0 \leq \int_{[0,1]^2 \cap \Phi(t,\supp h)} \frac{y_1y_2}{|y|^4} \omega(t,y) \,\ud y \leq C\| h \|_{L^{\infty}} N_0
\end{equation*}
hold for some $C>0$ and any sufficiently large $N_0$.

Since $g(\theta)=\varphi(\theta)-\varphi(\frac{\pi}{2}-\theta)$ takes both positive values and negative values, the support of $\omega_0$ can be decomposed into the dichotomous annular strips that are mutually disjoint and simply connected, say
\begin{equation*}
    \supp(m(r,\theta))=\bigcup_{n\geq N_0}\supp (\phi_n(r) g(\theta)).
\end{equation*}
We denote by $D^{(n)}$ the $n$-th dichotomous annular strip, $\supp (\phi_n(r) g(\theta))$. We further decompose each $D^{(n)}$ as $D^{(n)} = D_{+}^{(n)} \cup D_{-}^{(n)}$, where $D_{+}^{(n)}$ and $D_{-}^{(n)}$ stand for the positive part (meaning that $\omega_0(x) \geq 0$ when $x\in D_{+}^{(n)}$) and the negative part, respectively, of $\supp (\phi_n(r) g(\theta)) \subset [0,1]^2$.

\subsection{Short time dynamics}
We claim that each flow map starting from the $n$-th annular strip runs through a \emph{hyperbolic trajectory} on the time interval $[0,\delta]$ within the margin of some harmless error. One may compare this with Claim I in \cite{JK}.

\begin{lemma}[Short time dynamics]\label{lem_SD}
    Let $x \in D^{(n)}$ for $n \geq N_0$. Then for any $x' \in \cup_{\ell > n} D^{(\ell)}$,
    \begin{equation}\label{order_est}
        2\Phi_1(t,x') \leq \Phi_1(t,x) 
    \end{equation} holds on the time interval $[0,\delta]$. Moreover, there exist an increasing function $\gamma^{(n)} : [0,\delta] \to [0,\infty)$ and a vector valued function $\eta^{x} : [0,\delta] \to \bbR^2$ satisfying the initial conditions $\gamma^{(n)}(0) = \eta^{x}_1(0) = \eta^{x}_2(0) = 0$
such that the hyperbolic scenario
\begin{equation}\label{hyp_trj}
        \Phi(t,x) = (e^{\gamma^{(n)}(t) +\eta_1^{x}(t)} x_1, e^{-\gamma^{(n)}(t) +\eta_2^{x}(t)} x_2)
    \end{equation}
    holds on the time interval $[0,\delta]$, equipped with the bounds $\frac{\ud}{\ud t} \gamma^{(n)}(t) \geq \frac{N_0}{C}$ and $|\frac{\ud}{\ud t} \eta^{x}(t)| \leq C$.
\end{lemma}
\begin{remark}
    The main flow is governed by $\gamma^{(n)}(t)$ while one can view $\eta^{x}(t)$  as the small error.
\end{remark}

\begin{remark}\label{rmk_log_error}
    This observation means that the velocity field for the particles starting from $\cup_{n \geq N_0} D^{(n)}$ does not have the log-error term in \eqref{error_2} of our key lemma on the time interval $[0,\delta]$.
\end{remark}
\begin{proof}
    We begin with $x \in D^{(n)}$, $n = N_0$. We first show \eqref{order_est} and \eqref{hyp_trj} on the time interval $[0,T_n] \subset [0,\delta]$ where $\Phi(t,D^{(n)}) \subset \{ \frac{\pi}{36} \leq \tht \leq \frac{17\pi}{36} \}$ holds, and extend it to $[0,\delta]$ using the continuity argument. By the key lemma we have on the interval $[0,T_n]$ $$\frac{\ud}{\ud t}\gamma^{(n)}(t) = \frac{4}{\pi} \int_{[0,1]^2 \cap \Phi(t,\supp h)} \frac{y_1y_2}{|y|^4} \omega(t,y) \,\ud y$$ and
    \begin{equation}\label{no_log_est}
        \frac{\ud}{\ud t} (\log \Phi_1(t,x) - \gamma^{(n)}(t)) = \frac{\ud}{\ud t} \eta_1^{x}(t), \qquad \frac{\ud}{\ud t} (\log \Phi_2(t,x) + \gamma^{(n)}(t)) = \frac{\ud}{\ud t} \eta_2^{x}(t),
    \end{equation}
    where $\frac{\ud}{\ud t} \gamma^{(n)}(t) \geq \frac{N_0}{C}$ and $|\frac{\ud}{\ud t} \eta^{x}(t)| \leq C$. Integrating them over time, we obtain for $j=1,2$ that $\Phi_j(t,x) = e^{\gamma^{(n)}(t)+\eta_j^{x}(t)} x_j$ with $\gamma^{(n)} \geq \frac{N_0}{C}t$ and $|\eta^{x}(t)| \leq Ct$. Furthermore, for any $x,x' \in D^{(n)}$, $\frac{\Phi_1(t,x)}{\Phi_1(t,x')} = e^{\eta_1^{x}(t) - \eta_1^{x'}(t)} \frac{x_1}{x'_1}$ and $\frac{\Phi_2(t,x)}{\Phi_2(t,x')} = e^{\eta_2^{x}(t) - \eta_2^{x'}(t)} \frac{x_2}{x'_2}$ hold. Since there exist some constants $C_1>1$ and $C_2>1$ not depending on $n$ such that $2 x_1 \geq C_1 x'_1$ and $2 x_2 \geq C_2 x'_2$, we have $2\Phi_1(t,x) \geq C_1 e^{-C\delta} \Phi_1(t,x')$ and $2\Phi_2(t,x) \geq C_2 e^{-C\delta} \Phi_2(t,x')$. Thus, if $\delta$ satisfies $(\sqrt{C_1} + \sqrt{C_2})e^{-C\delta} \geq 1$, then it follows that
    \begin{equation}\label{C12_est}
        2\Phi_1(t,x) \geq \sqrt{C_1} \Phi_1(t,x') \qquad \mbox{and} \qquad 2\Phi_2(t,x) \geq \sqrt{C_2} \Phi_2(t,x').
    \end{equation}

    Next, we show \eqref{order_est} on $[0,T_n]$. Let $\Psi_1(t) = \sup_{x' \in \cup_{\ell > n} D^{(\ell)}} \Phi_1(t,x')$. We note that there exists a constant $C_3 > 1$ not depending on $n$ such that $2C_3\Psi_1(t=0) \leq x_1$ whenever we take $\tht_0$ small enough. From the key lemma, we can see
    \begin{equation*}
        \frac{\ud}{\ud t} \left( \log \Psi_1(t,x) - \log \Phi_1(t) \right) \leq \frac{4}{\pi} \int_{[2\Psi_1(t), 1] \times [0,1] \setminus Q(2\Phi(t,x))} \frac{y_1y_2}{|y|^4} \omega(t,y) \,\ud y + C \| \omega_0 \|_{L^{\infty}}.
    \end{equation*}
    By the definition of $\Psi_1$ and \eqref{hyp_trj}, there holds
    \begin{equation}\label{A_est}
    \begin{aligned}
        \int_{[2\Psi_1(t), 1] \times [0,1] \setminus Q(2\Phi(t,x))} \frac{y_1y_2}{|y|^4} \omega(t,y) \,\ud y &\leq \int_{D_{+}^{(n)}} \frac{\Phi_1(t,y)\Phi_2(t,y)}{|\Phi(t,y)|^4} \omega_0(y) \,\ud y \leq e^{C\delta} A n^{-\beta},
    \end{aligned}
    \end{equation}
    where 
    \begin{equation}\label{def_A}
        A := \sup_{\tau \in (-\infty,\infty)} \int_{D_{+}^{(n)}} \frac{y_1y_2}{|(e^{\tau}y_1,e^{-\tau}y_2)|^4} \phi^{(n)}(y) \,\ud y.
    \end{equation} Here, $A$ does not depend on $n$, and so
    \begin{equation*}
        \frac{\ud}{\ud t} \left( \log \Psi_1(t) - \log \Phi_1(t,x) \right) \leq C.
    \end{equation*}
    Integrating it over time, we obtain $e^{-C\delta} 2C_3 \Psi_1(t) \leq e^{-C\delta} \frac{x_1}{\Psi(t=0)}\Psi_1(t) \leq \Phi_1(t,x)$. Therefore, for sufficiently small $\delta>0$ with $e^{-C\delta} C_3 \geq \sqrt{C_3}$, we obtain 
    \begin{equation}\label{C3_est}
        2\sqrt{C_3} \Psi_1(t) \leq \Phi_1(t,x),
    \end{equation} which implies \eqref{order_est} on $[0,T_n]$. 
    
    Next, we show that $T_n$ can be replaced by $\delta$. By the previous estimates, we can find $T \in (T_n,\delta]$ such that \eqref{C12_est} and \eqref{C3_est} hold, and thus, we can verify that $R(2\Phi(t,x)) \cap \Phi(t,\supp \omega_0) = \emptyset$ and $Q(2\Phi(t,x)) \cap \Phi(t,\supp \omega_0) = \supp h$ on $[0,T]$. This gives \eqref{no_log_est} with the same upper bounds of $\gamma^{(n)}(t)$ and $\eta^{x}(t)$. Then \eqref{hyp_trj} and \eqref{C12_est} follow on the interval $[0,T]$. By repeating the above procedure, we have \eqref{C3_est} for all $t \in [0,T]$. Therefore, we can take $T=\delta$.  

    We consider $n = N_0+1$. By \eqref{C3_est}, there holds \begin{equation}\label{gamma_est}
    \frac{\ud}{\ud t}\gamma^{(n)}(t) = \frac{4}{\pi} \int_{[0,1]^2 \cap \Phi(t,\supp h)} \frac{y_1y_2}{|y|^4} \omega(t,y) \,\ud y + \frac{4}{\pi} \sum_{\ell = N_0}^{n-1} \int_{\Phi(t,D^{(\ell)})} \frac{y_1y_2}{|y|^4} \omega(t,y) \,\ud y.
    \end{equation} If the positivity condition $$\int_{\Phi(t,D^{(\ell)})} \frac{y_1y_2}{|y|^4} \omega(t,y) \,\ud y > 0$$ holds on $[0,\delta]$, then repeating the previous estimates for $n=N_0$ delivers the same results. Thus, it suffices to show that $$\int_{D^{(\ell)}} \frac{\Phi_1(t,y) \Phi_2(t,y)}{|\Phi(t,y)|^4} \omega_0(y) \,\ud y > 0.$$ Using \eqref{hyp_trj}, we have 
    \begin{gather*}
        \int_{D^{(\ell)}} \frac{\Phi_1(t,y) \Phi_2(t,y)}{|\Phi(t,y)|^4} \omega_0(y) \,\ud y =  \int_{D_{+}^{(\ell)}} \frac{e^{\eta_1^{y}(t) + \eta_2^{y}(t)} y_1y_2}{\left|(e^{\gamma^{(\ell)}(t) +\eta_1^{y}(t)} y_1, e^{-\gamma^{(\ell)}(t) +\eta_2^{y}(t)} y_2)\right|^4} \omega_0(y) \,\ud y \\
        + \int_{D_{-}^{(\ell)}} \frac{e^{\eta_1^{y}(t) + \eta_2^{y}(t)} y_1y_2}{\left|(e^{\gamma^{(\ell)}(t) +\eta_1^{y}(t)} y_1, e^{-\gamma^{(\ell)}(t) +\eta_2^{y}(t)} y_2)\right|^4} \omega_0(y) \,\ud y.
    \end{gather*} We note for $y \in D_{+}^{(\ell)}$ that
    \begin{equation*}\label{y_est}
    \begin{gathered}
        \frac{\ud}{\ud t}\left|(e^{\gamma^{(\ell)}(t) +\eta_1^{y}(t)} y_1, e^{-\gamma^{(\ell)}(t) +\eta_2^{y}(t)} y_2)\right|^2 = 2\partial_t (\gamma^{(\ell)}(t) + \eta_1^{y}(t))e^{2(\gamma^{(\ell)}(t) + \eta_1^{y}(t))} y_1^2 \\
        + 2\partial_t (-\gamma^{(\ell)}(t) + \eta_2^{y}(t))e^{2(-\gamma^{(\ell)}(t) + \eta_2^{y}(t))} y_2^2.
    \end{gathered}
    \end{equation*}
    On the other hand, for $y' = (y_2,y_1)$, it follows that
    \begin{equation*}\label{y'_est}
    \begin{gathered}
        \frac{\ud}{\ud t}\left|(e^{\gamma^{(\ell)}(t) +\eta_1^{y'}(t)} y_2, e^{-\gamma^{(\ell)}(t) +\eta_2^{y'}(t)} y_1)\right|^2 = 2\partial_t (\gamma^{(\ell)}(t) + \eta_1^{y'}(t))e^{2(\gamma^{(\ell)}(t) + \eta_1^{y'}(t))} y_2^2 \\
        + 2\partial_t (-\gamma^{(\ell)}(t) + \eta_2^{y'}(t))e^{2(-\gamma^{(\ell)}(t) + \eta_2^{y'}(t))} y_1^2.
    \end{gathered}
    \end{equation*}
    Hence, $y_2 > y_1$ and $\frac{\ud}{\ud t} \gamma^{(\ell)} (t) >> |\frac{\ud}{\ud t} \eta^{y}(t)|$ imply
    \begin{equation*}
        \frac{\ud}{\ud t}\left|(e^{\gamma^{(\ell)}(t) +\eta_1^{y}(t)} y_1, e^{-\gamma^{(\ell)}(t) +\eta_2^{y}(t)} y_2)\right|^2 \leq \frac{\ud}{\ud t}\left|(e^{\gamma^{(\ell)}(t) +\eta_1^{y'}(t)} y_2, e^{-\gamma^{(\ell)}(t) +\eta_2^{y'}(t)} y_1)\right|^2
    \end{equation*} and
    \begin{equation*}
        \left|(e^{\gamma^{(\ell)}(t) +\eta_1^{y}(t)} y_1, e^{-\gamma^{(\ell)}(t) +\eta_2^{y}(t)} y_2)\right|^2 \leq \left|(e^{\gamma^{(\ell)}(t) +\eta_1^{y'}(t)} y_2, e^{-\gamma^{(\ell)}(t) +\eta_2^{y'}(t)} y_1)\right|^2
    \end{equation*} on the interval $[0,\delta]$. Combining the above estimates, we can deduce $$\frac{\ud}{\ud t} \frac{e^{\eta_1^{y}(t) + \eta_2^{y}(t)} y_1y_2}{\left|(e^{\gamma^{(\ell)}(t) +\eta_1^{y}(t)} y_1, e^{-\gamma^{(\ell)}(t) +\eta_2^{y}(t)} y_2)\right|^4} \geq \frac{\ud}{\ud t} \frac{e^{\eta_1^{y'}(t) + \eta_2^{y'}(t)} y_1y_2}{\left|(e^{\gamma^{(\ell)}(t) +\eta_1^{y'}(t)} y_2, e^{-\gamma^{(\ell)}(t) +\eta_2^{y'}(t)} y_1)\right|^4}.$$ With $|y|=|y'|$, we can deduce the claim. This completes the proof.
\end{proof}

\subsection{Positivity of key integrals}

The following lemma says that the key integral in Lemma~\ref{lem_key} is dominated by the positive parts of the vorticity right after the initial time.

\begin{lemma}[Positivity]\label{lem_C1_inflation}
    There exists a constant $C > 0$ not depending on $n$ such that 
    \begin{equation}\label{C1_inflation}
        \int_{D^{(n)}} \frac{\Phi_1(t,y)\Phi_2(t,y)}{\left|\Phi(t,y)\right|^4} \omega_0(y) \,\ud y \geq (1-e^{Ct - 2\gamma^{(n)}(t)}) \int_{D_{+}^{(n)}} \frac{\Phi_1(t,y)\Phi_2(t,y)}{\left|\Phi(t,y)\right|^4}  \omega_0(y) \,\ud y
    \end{equation}
    as long as $\frac{\Phi_2(t,y)^2}{\Phi_1(t,y)^2} \geq \frac{5}{3}$ holds for all $y \in D_{+}^{(n)}$.
\end{lemma}
\begin{remark}
    By \eqref{hyp_trj}, we have $\frac{\Phi_2(t,y)^2}{\Phi_1(t,y)^2} \geq 3e^{-C\delta} e^{-4\gamma^{(n)}(t)}$ for $y \in D_{+}^{(n)}$. Thus, \eqref{C1_inflation} holds whenever $e^{-4\gamma^{(n)}(t)} \geq \frac{2}{3}$ is satisfied.
\end{remark}
\begin{proof}
    Let $n \geq N_0$. Let $y \in D_{+}^{(n)}$ and $y'=(y_2,y_1) \in D_{-}^{(n)}$. By \eqref{hyp_trj} we can write $$\Phi(t,y') = (2e^{\gamma^{(n)}(t) + \eta_1^{y'}(t) - \eta_2^{y}(t)} \Phi_2(t,y), e^{-2\gamma^{(n)}(t) + \eta_2^{y'}(t) - \eta_1^{y}(t)} \Phi_1(t,y)).$$ Then, using the bound of the $\eta$ function, we have
    \begin{gather*}
        \int_{D_{-}^{(n)}} \frac{\Phi_1(t,y')\Phi_2(t,y')}{\left|\Phi(t,y')\right|^4} \omega_0(y') \,\ud y' \\
        = - \int_{D_{+}^{(n)}} \frac{e^{\eta_1^{y'}(t) - \eta_1^{y}(t)+\eta_2^{y'}(t) - \eta_2^{y}(t)}\Phi_1(t,y)\Phi_2(t,y)}{\left|(e^{2\gamma^{(n)}(t) + \eta_1^{y'}(t) - \eta_2^{y}(t)} \Phi_2(t,y), e^{-2\gamma^{(n)}(t) + \eta_2^{y'}(t) - \eta_1^{y}(t)} \Phi_1(t,y))\right|^4} \omega_0(y) \,\ud y \\
        \geq - e^{Ct - 2\gamma^{(n)}(t)} \int_{D_{+}^{(n)}} \frac{\Phi_1(t,y)\Phi_2(t,y)}{\left|(e^{\frac{3}{2}\gamma^{(n)}(t)} \Phi_2(t,y), e^{-\frac{5}{2}\gamma^{(n)}(t)} \Phi_1(t,y))\right|^4} \omega_0(y) \,\ud y.
    \end{gather*}
    We note that $$\left|(e^{\frac{3}{2}\gamma^{(n)}(t)} \Phi_2(t,y), e^{-\frac{5}{2}\gamma^{(n)}(t)} \Phi_1(t,y))\right|^2 \geq \left| \Phi(t,y) \right|^2$$ if and only if $$\left( \frac{\Phi_2(t,y)}{\Phi_1(t,y)} \right)^2 \geq \frac{1-e^{-5\gamma^{(n)}(t)}}{e^{3\gamma^{(n)}(t)}-1}.$$ Using that the right-hand side is less than $\frac{5}{3}$, we obtain
    \begin{gather*}
        \int_{D_{-}^{(n)}} \frac{\Phi_1(t,y')\Phi_2(t,y')}{\left|\Phi(t,y')\right|^4} \omega_0(y') \,\ud y' \geq -e^{Ct - 2\gamma^{(n)}(t)} \int_{D_{+}^{(n)}} \frac{\Phi_1(t,y)\Phi_2(t,y)}{\left|\Phi(t,y)\right|^4} \omega_0(y') \,\ud y.
    \end{gather*}
    This completes the proof.
\end{proof}

\subsection{Proof of Theorem~\ref{thm1}}

\begin{proof}[Proof of Theorem~\ref{thm1}]
At first, we consider $\beta \in (\frac{1}{2},1]$ and prove that $u \not\in L^{\infty}([0,\delta];C^1(\bbT^2))$ for any $\delta>0$. Let $\tht_0$ be a small constant and $N_0$ be a large constant to satisfy all the above estimates. We fix $\delta$ to be small enough that $N_0\delta$ is sufficiently small, and let $T \in [0,\delta]$. For any $x \in D^{(n)}$ with $n \geq N_0$, Remark~\ref{rmk_log_error} gives us that $$\left| \frac{\ud}{\ud t} \log \Phi_1(t,x) - \frac{\ud}{\ud t} \gamma^{(n)}(t) \right| + \left| \frac{\ud}{\ud t} \log \Phi_2(t,x) + \frac{\ud}{\ud t} \gamma^{(n)}(t) \right| \leq C,$$ where $$\frac{\ud}{\ud t} \gamma^{(n)}(t) = \frac{4}{\pi} \int_{[0,1]^2 \cap \Phi(t,\supp h)} \frac{y_1y_2}{|y|^4} \omega(t,y) \,\ud y + \frac{4}{\pi} \sum_{\ell = N_0}^{n-1} \int_{\Phi(t,D^{(\ell)})} \frac{y_1y_2}{|y|^4} \omega(t,y) \,\ud y.$$ Since $\Phi(t,\supp h) \sim \supp h$ on $[0,\delta]$, it is clear that $$\frac{1}{2} \int_{[0,1]^2 \cap \supp h} \frac{y_1y_2}{|y|^4} \omega_0(y) \,\ud y \leq \int_{[0,1]^2 \cap \Phi(t,\supp h)} \frac{y_1y_2}{|y|^4} \omega(t,y) \,\ud y \leq 2 \int_{[0,1]^2 \cap \supp h} \frac{y_1y_2}{|y|^4} \omega_0(y) \,\ud y.$$
Recalling \eqref{def_A}, we can estimate for $\beta<1$ as $$\sum_{\ell = N_0}^{n-1} \int_{\Phi(t,D^{(\ell)})} \frac{y_1y_2}{|y|^4} \omega(t,y) \,\ud y \leq e^{C\delta} \sum_{\ell=N_0}^{n-1} A \ell^{-\beta} \leq e^{C\delta} \frac{A}{1-\beta} n^{1-\beta},$$ and for $\beta=1$ as $$\sum_{\ell = N_0}^{n-1} \int_{\Phi(t,D^{(\ell)})} \frac{y_1y_2}{|y|^4} \omega(t,y) \,\ud y \leq e^{C\delta} \sum_{\ell=N_0}^{n-1} A \ell^{-1} \leq e^{C\delta} A \log n.$$
The above estimates imply that there exists a constant $C>0$ not depending on $\beta$ and $\ell$ such that for $\xi^{(\ell)} = \frac{1-\beta}{C} \frac{1}{\ell^{1-\beta}}$ with $\beta <1 $ and $\xi^{(\ell)} = \frac{1}{C} \frac{1}{\log \ell}$ with $\beta = 1$, we have $\Phi(t,D^{(\ell)}) \sim \Phi(t',D^{(\ell)})$ whenever $|t-t'| \leq \xi^{(\ell)}$ on the interval $[0,\delta]$. On the other hand, \eqref{C1_inflation} gives that 
\begin{align*}
    \int_{\Phi(t, D^{(\ell)})} \frac{y_1y_2}{|y|^4} \omega(t,y) \,\ud y &\geq (1-e^{CT-\frac{N_0}{C}T}) \int_{\Phi(t, D_{+}^{(\ell)})} \frac{y_1y_2}{|y|^4} \omega(t,y) \,\ud y \\
    &\geq (1-e^{-\frac{N_0}{C} \frac{T}{2}}) \int_{\Phi(t, D_{+}^{(\ell)})} \frac{y_1y_2}{|y|^4} \omega(t,y) \,\ud y
\end{align*}
for $t \in [T^{(\ell)},T]$, where $T^{(\ell)} \in (0,\frac{T}{2}]$ with $\gamma^{(\ell)}(T^{(\ell)}) = \frac{N_0}{C} \frac{T}{2}$ so that $1-e^{CT^{(\ell)}-2\gamma^{(\ell)}(T^{(\ell)})} \geq 1- e^{CT-\frac{N_0}{C}T}$. We used the largeness of $N_0$ in the last inequality. Note that $\frac{\Phi_2(T^{(\ell)},y)^2}{\Phi_1(T^{(\ell)},y)^2} > \frac{5}{3}$ is satisfied for all $y \in D_{+}^{(\ell)}$ from the smallness of $N_0T$. 

Our claim is that there exist $y \in D_{+}^{(n)}$ and $n \geq N_0$ such that $$\frac{\Phi_2(T,y)^2}{\Phi_1(T,y)^2} \leq 2.$$ If the claim holds, then $$3 e^{-C\delta} e^{-4\gamma^{(n)}(T)} \leq  \frac{y_2^2}{y_1^2} e^{-C\delta} e^{-4\gamma^{(n)}(T)} \leq 2.$$
Thus, for sufficiently small $\delta$, we have a positive constant $\varepsilon = \frac{1}{4} (\log \frac{3}{2} - C\delta)$ such that $$\varepsilon \leq \gamma^{(n)}(T) - C\delta \leq \log \frac{\Phi_1(t,y)}{y_1} = \int_0^T \frac{u_1(t,\Phi(t,y))}{\Phi_1(t,y)} \,\ud t.$$ The mean-value theorem in time variable says that there exists $t \in (0,T]$ satisfying
\begin{equation}\label{blowup}
    \frac{\varepsilon}{T} \leq \frac{u_1(t,\Phi(t,y))}{\Phi_1(t,y)}.
\end{equation}
 Since $T \in (0,\delta]$ is arbitrary, we conclude that $$\sup_{t \in (0,\delta)} \| u(t,\cdot) \|_{C^1} = \infty.$$ This would finish the proof of the $C^1$ part of Theorem~\ref{thm1}.


It remains to verify the claim. To that end, we divide the proof into the two cases $T<T^{(\ell)} + \xi^{(\ell)}$ and $T^{(\ell)} + \xi^{(\ell)} \leq T$. Note that sufficiently large $\ell$ with $\xi^{(\ell)} \leq \frac{T}{2}$ falls onto the second case. Hence, we have the diverging point $\ell_0$ with $\ell_0 \sim (\frac{1-\beta}{T})^{\frac{1}{1-\beta}}$ for $\beta < 1$ and $\ell_0 \sim e^{\frac{1}{T}}$ for $\beta = 1$. In the first case, we only use a rough bound
\begin{equation*}
    \int_0^T \int_{\Phi(t,D^{(\ell)})} \frac{y_1y_2}{|y|^4} \omega(t,y) \,\ud y \ud t \geq 0.
\end{equation*} In the other case, we can see
\begin{align*}
    \int_{\Phi(t,D^{(\ell)})} \frac{y_1y_2}{|y|^4} \omega(t,y) \,\ud y &\geq (1-e^{-\frac{N_0}{C}\frac{T}{2}}) \int_{\Phi(t,D_{+}^{(\ell)})} \frac{y_1y_2}{|y|^4} \omega(t,y) \,\ud y \\
    &= (1-e^{-\frac{N_0}{C}\frac{T}{2}}) \int_{D_{+}^{(\ell)}} \frac{\Phi_1(t,y)\Phi_2(t,y)}{|\Phi(t,y)|^4} \omega_0(y) \,\ud y \\
    &\geq \frac{1}{2} (1-e^{-\frac{N_0}{C}\frac{T}{2}}) \int_{D_{+}^{(\ell)}} \frac{\Phi_1(T^{(\ell)},y)\Phi_2(T^{(\ell)},y)}{|\Phi(T^{(\ell)},y)|^4} \omega_0(y) \,\ud y
\end{align*}
on the interval $[T^{(\ell)}, T^{(\ell)} + \xi^{(\ell)}]$. Moreover, due to Lemma~\ref{lem_SD} and $\gamma^{(\ell)}(T^{(\ell)}) = \frac{N_0}{C}\frac{T}{2}$, there holds 
\begin{equation*}
    \begin{aligned}
    &\frac{1}{2} (1-e^{-\frac{N_0}{C}\frac{T}{2}}) \int_{D_{+}^{(\ell)}} \frac{\Phi_1(T^{(\ell)},y)\Phi_2(T^{(\ell)},y)}{|\Phi(T^{(\ell)},y)|^4} \omega_0(y) \,\ud y \\
    &\hphantom{\qquad\qquad\qquad} \geq \frac{1}{2} (1-e^{-\frac{N_0}{C}\frac{T}{2}}) e^{-C\delta} \ell^{-\beta} \int_{D_{+}^{(\ell)}} \frac{y_1y_2}{\left|(e^{\frac{N_0}{C}\frac{T}{2}}y_1,e^{-\frac{N_0}{C}\frac{T}{2}}y_2)\right|^4} \phi^{(\ell)}(y) \,\ud y \\
    &\hphantom{\qquad\qquad\qquad} \geq \frac{1}{2} (1-e^{-\frac{N_0}{C}\frac{T}{2}}) e^{-C\delta} \ell^{-\beta} \int_{D_{+}^{(\ell)}} \frac{y_1y_2}{|y|^4} \phi^{(\ell)}(y) \,\ud y \\
    &\hphantom{\qquad\qquad\qquad} \geq \frac{N_0T}{C} \ell^{-\beta}.
\end{aligned}
\end{equation*} We have used 
the Taylor expansion with the smallness of $N_0T$ in the last inequality. As a result,
\begin{equation*}
    \int_0^T \int_{\Phi(t,D^{(\ell)})} \frac{y_1y_2}{|y|^4} \omega(t,y) \,\ud y \,\ud t \geq \int_{T^{(\ell)}}^{T^{(\ell)} + \xi^{(\ell)}} \int_{\Phi(t,D^{(\ell)})} \frac{y_1y_2}{|y|^4} \omega(t,y) \,\ud y \,\ud t \geq \frac{N_0T}{C} \ell^{-\beta} \xi^{(\ell)}.
\end{equation*} Combining the above estimates, for $x \in D^{(n)}$ with sufficiently large $n$, we have that
\begin{equation*}
    \log \frac{\Phi_1(T,x)}{x_1} = \int_0^T \frac{\ud}{\ud t} \log \Phi_1(t,x) \,\ud t \geq \frac{N_0T}{C} \sum_{\ell = \ell_0}^{n-1} \ell^{-\beta} \xi^{(\ell)} - C\delta.
\end{equation*}
For $\beta < 1$, we see $$\sum_{\ell = \ell_0}^{n-1} \ell^{-\beta} \xi^{(\ell)} = \frac{1-\beta}{C} \sum_{\ell = \ell_0}^{n-1} \ell^{-1} \geq \frac{1-\beta}{C} \log \frac{n}{\ell_0} \geq \frac{1-\beta}{C} \log \frac{n}{C} \left( \frac{T}{1-\beta} \right)^{\frac{1}{1-\beta}}$$ and 
\begin{equation*}
    \log \frac{\Phi_1(T,x)}{x_1} \geq \frac{N_0T}{C} \frac{1-\beta}{C} \log \frac{n}{C} \left( \frac{T}{1-\beta} \right)^{\frac{1}{1-\beta}} - C\delta.
\end{equation*} Since we can show
\begin{equation}\label{Phi2_est}
    \log \frac{\Phi_2(T,x)}{x_2} \leq -\frac{N_0T}{C} \frac{1-\beta}{C} \log \frac{n}{C} \left( \frac{T}{1-\beta} \right)^{\frac{1}{1-\beta}} + C\delta
\end{equation} in a similar way, we achieve the claim for $n \gtrsim e^{\frac{1}{T}} T^{-\frac{1}{1-\beta}}$. On the other hand, we have for $\beta = 1$ $$\sum_{\ell = \ell_0}^{n-1} \ell^{-1} \xi^{(\ell)} = \frac{1}{C} \sum_{\ell = \ell_0}^{n-1} \ell^{-1} \frac{1}{\log \ell} \geq \frac{1}{C} \log \log \frac{n}{\ell_0} \geq \frac{1}{C} \log \log \frac{n}{C} e^{-\frac{1}{T}}.$$ This implies both \begin{equation*}
    \log \frac{\Phi_1(T,x)}{x_1} \geq \frac{N_0T}{C} \log \log \frac{n}{C} e^{-\frac{1}{T}} - C\delta
\end{equation*}
and
\begin{equation*}
    \log \frac{\Phi_2(T,x)}{x_2} \leq - \frac{N_0T}{C} \log \log \frac{n}{C} e^{-\frac{1}{T}} + C\delta.
\end{equation*}
Then the claim is established when $n \gtrsim e^{e^{\frac{1}{T}}} e^{\frac{1}{T}}$.

To prove the $H^2$ part of Theorem~\ref{thm1}, we claim that
if $\beta$ is sufficiently close to $\frac{1}{2}$, then $u \not \in L^{\infty}([0,\delta];H^2(\bbT^2))$ for any given $\delta > 0$. Once the claim is verified, then the freedom of the choice of $\beta\in(\frac{1}{2},1]$ allows us to witness the dual blow-up in both $C^1$ and $H^2$, closing the proof of Theorem~\ref{thm1}.

It remains to show the claim. We need to estimate \eqref{Phi2_est} to obtain an upper bound which is independent of $T$ and $\beta \sim \frac{1}{2}$. We redefine $T^{(\ell)} > 0$ by the the first time when $\frac{\Phi_2(T^{(\ell)},y)^2}{\Phi_1(T^{(\ell)},y)^2} = 2$ is satisfied for some $y \in D^{(\ell)}$. Then, $e^{-4\gamma^{(n)}(T^{(\ell)})} \sim 2 \frac{y_1^2}{y_2^2} < C_4$ for some $C_4 < 1$ by the smallness of $|e^{\eta^{y}(t)}| \leq e^{C\delta}$ and $\tht_0$. To guarantee the existence of $T^{(\ell)}$, we only consider sufficiently large $\ell$. Since we have $\lim_{\ell \to \infty} T^{(\ell)} = 0$ from the previous result, we can take $\ell_0 = \ell_0(T)$ such that $T^{(\ell)} + \xi^{(\ell)} < T$ for all $\ell \geq \ell_0$. 

Using the above information, we deduce from \eqref{C1_inflation} that
\begin{align*}
    \int_{D^{(\ell)}} \frac{\Phi_1(t,y)\Phi_2(t,y)}{\left|\Phi(t,y)\right|^4} \omega_0(y) \,\ud y &\geq (1-e^{CT^{(\ell)} - 2\gamma^{(\ell)}(T^{(\ell)})}) \int_{D_{+}^{(\ell)}} \frac{\Phi_1(t,y)\Phi_2(t,y)}{\left|\Phi(t,y)\right|^4}  \omega_0(y) \,\ud y \\
    &\geq (1-\sqrt{C_4} e^{C\delta}) \ell^{-\beta} \int_{D_{+}^{(\ell)}} \frac{\Phi_1(t,y)\Phi_2(t,y)}{\left|\Phi(t,y)\right|^4}  \phi^{(\ell)} (y) \,\ud y
\end{align*}  on the time interval $[T^{(\ell)},T^{(\ell)} + \xi^{(\ell)}]$. Since $\Phi(t,y) \sim \Phi(T^{(\ell)},y)$ holds, we have $$\int_{D_{+}^{(\ell)}} \frac{\Phi_1(t,y)\Phi_2(t,y)}{\left|\Phi(t,y)\right|^4} \phi^{(\ell)} (y) \,\ud y \geq e^{-C\delta} \int_{D_{+}^{(\ell)}} \frac{y_1y_2}{|y|^4}  \phi^{(\ell)}(y) \,\ud y,$$
and so there exists a constant $C>0$ not depending on $T$, $\ell$, and $\beta$ such that
\begin{align*}
    \int_{D^{(\ell)}} \frac{\Phi_1(t,y)\Phi_2(t,y)}{\left|\Phi(t,y)\right|^4} \omega_0(y) \,\ud y \geq \frac{1}{C} \ell^{-\beta}.
\end{align*}
Let $x \in D^{(n)}$ with $n > \ell_0$. Recalling that $\xi^{(\ell)} \sim \frac{1}{\ell^{1-\beta}}$ uniformly-in-$\ell$ for $\beta$ close enough to $\frac{1}{2}$, we perform the estimation as
\begin{equation*}
    \int_0^T \int_{\Phi(t,D^{(\ell)})} \frac{y_1y_2}{|y|^4} \omega(t,y) \,\ud y \,\ud t \geq \int_{T^{(\ell)}}^{T^{(\ell)} + \xi^{(\ell)}} \int_{\Phi(t,D^{(\ell)})} \frac{y_1y_2}{|y|^4} \omega(t,y) \,\ud y \,\ud t \geq \frac{1}{C} \ell^{-1}.
\end{equation*}
This yields $$\log \frac{\Phi_2(T,x)}{x_2} = \int_0^T \frac{\ud}{\ud t} \log \Phi_2(t,x) \,\ud t \leq - \frac{1}{C}\sum_{\ell = \ell_0}^{n-1} \ell^{-1} + C\delta \leq -\frac{1}{C} \log \frac{n}{\ell_0} + C\delta,$$ 
and thus, we obtain $$\frac{\Phi_2(T,x)}{x_2} \leq C^* n^{-\frac{1}{C_*}}.$$ 
It is worth emphasizing that both positive constants $C^*$ and $C_*$ on the right-hand side do not depend on $n$, $T$ and $\beta \sim \frac{1}{2}$. Employing Hardy's inequality, we observe that
\begin{equation*}
    \| \omega(T) \|_{H^1(\bbT^2)}^2 \geq \frac{1}{C} \left\| \frac{\omega}{x_2}(T) \right\|_{L^2(\bbT^2)}^2 \geq \sum_{n \geq \ell_0} \left\| \frac{\omega}{x_2}(T) \right\|_{L^2(\Phi(T,D^{(n)}))}^2.
\end{equation*}
A simple inequality $$\int_{\Phi(T,D^{(n)})} \left| \frac{\omega(T,x)}{x_2} \right|^2 \,\ud x = \int_{D^{(n)}} \left| \frac{\omega_0(x)}{x_2} \right|^2 \left| \frac{x_2}{\Phi_2(T,x)} \right|^2 \,\ud x \geq \inf_{x \in D^{(n)}} \left( \frac{x_2}{\Phi_2(T,x)} \right)^2 \int_{D^{(n)}} \left| \frac{\omega_0(x)}{x_2} \right|^2 \,\ud x,$$
holds, which leads to our blow-up rate estimate with respect to $n$ as
\begin{equation}\label{H2_dist}
    \left\| \frac{\omega}{x_2}(T) \right\|_{L^2(\Phi(T,D^{(n)}))}^2 \geq C n^{\frac{2}{C_*}-2\beta}.
\end{equation} Fixing $\beta < \frac{1}{2} + \frac{1}{C_*}$, we clearly obtain $\|  \omega(T) \|_{H^1(\bbT^2)}^2 = \infty$. Since the lower bound does not depend on $T$, we conclude that $$\sup_{t \in (0,T)} \| u(t) \|_{H^2(\bbT^2)} = \infty, \qquad \forall T \in (0,\delta].$$ This completes the proof. \end{proof}

\section*{Acknowledgment}
The authors greatly thank In-Jee Jeong for suggesting this research and for his insightful comments.

\bibliographystyle{amsplain}


\end{document}